\documentclass[12pt,pstricks]{article}
\usepackage{epstopdf}
\usepackage[dvipsnames]{xcolor}
\usepackage{pst-plot}
\psset{xunit=15mm}
\usepackage{tikz}
\usetikzlibrary{matrix,arrows,decorations.pathmorphing}
\usepackage{soul}
\usepackage{booktabs}
\usepackage{geometry}
\usepackage{graphicx}

\usepackage{hyperref}
\usepackage{empheq}
\usepackage{nameref}
\usepackage[shortlabels]{enumitem}
\usepackage[doc]{optional}
\DeclareGraphicsExtensions{.jpg}
\usepackage{float}
\usepackage{graphicx}
\definecolor{labelkey}{rgb}{0,0.08,0.45}
\definecolor{refkey}{rgb}{0,0.6,0.0}
\definecolor{Brown}{rgb}{0.45,0.0,0.05}
\definecolor{lime}{rgb}{0.00,0.8,0.0}
\definecolor{lblue}{rgb}{0.5,0.5,0.99}
\DeclareGraphicsExtensions{.eps}

 \usepackage{mathpazo}

\usepackage{stmaryrd}
\usepackage{amssymb}

\newcommand{\seppfive}{\setlength{\itemsep}{-5pt}}

\makeatletter
\def\namedlabel#1#2{\begingroup
   \def\@currentlabel{#2}%
   \label{#1}\endgroup
}
\makeatother

\newcommand*{\tran}{^{\mkern-1.5mu\mathsf{T}}}
\providecommand{\siff}{\Leftrightarrow}

\newcommand{\menge}[2]{\big\{{#1}~\big |~{#2}\big\}}

\newcommand{\To}{\ensuremath{\rightrightarrows}}

\newcommand{\fenv}[1]%
{\ensuremath{\,\overrightarrow{\operatorname{env}}_{#1}}}
\newcommand{\benv}[1]%
{\ensuremath{\,\overleftarrow{\operatorname{env}}_{#1}}}

\newcommand{\RR}{\ensuremath{\mathbb R}}

\newcommand{\NN}{\ensuremath{\mathbb N}}

\newcommand{\dom}{\ensuremath{\operatorname{dom}}}
\newcommand{\argmin}{\ensuremath{\operatorname{argmin}}}
\newcommand{\prox}{\ensuremath{\operatorname{Prox}}}

\newcommand{\ran}{\ensuremath{\operatorname{ran}}}
\newcommand{\zer}{\ensuremath{\operatorname{zer}}}

\newcommand{\Id}{\ensuremath{\operatorname{Id}}}

\providecommand{\pars}{\partial_{\mathmakebox[0.4em][l]{\textrm{\#}}}}

{\begin{list}{}{%
\settowidth{\labelwidth}{\textrm{#1~}}%
\setlength{\leftmargin}{\labelwidth+\labelsep}}}
{\end{list}}
\usepackage{amsthm}
\usepackage[capitalize, nameinlink]{cleveref}
\crefname{enumi}{}{}
\crefname{equation}{}{equations}
\crefname{chapter}{Appendix}{chapters}
\crefname{item}{}{items}
\newtheorem{theorem}{Theorem}[section]
\newtheorem{lemma}[theorem]{Lemma}
\newtheorem{corollary}[theorem]{Corollary}
\newtheorem{proposition}[theorem]{Proposition}
\newtheorem{prop}[theorem]{Proposition}
\newtheorem{definition}[theorem]{Definition}
\newtheorem{defn}[theorem]{Definition}
\newtheorem{thm}[theorem]{Theorem}
\newtheorem{example}[theorem]{Example}
\newtheorem{ex}[theorem]{Example}
\newtheorem{fact}[theorem]{Fact}
\newtheorem{remark}[theorem]{Remark}
\newtheorem{rem}[theorem]{Remark}

\providecommand{\norm}[1]{\lVert#1\rVert}
\providecommand{\normsq}[1]{\lVert#1\rVert^2}
\providecommand{\bk}[1]{\left(#1\right)}

\providecommand{\innp}[1]{\langle#1\rangle}

\providecommand{\LA}{\Leftarrow}
\providecommand{\RA}{\Rightarrow}

\providecommand{\grad}{\nabla}

\providecommand{\lam}{\lambda}
\providecommand{\RR}{\mathbb{R}}

\providecommand{\ran}{\operatorname{ran}}

\providecommand{\dom}{\operatorname{dom}}

\newcommand{\fix}{\ensuremath{\operatorname{Fix}}}
\newcommand{\conic}{\text{conically nonexpansive}}

\providecommand{\gra}{\operatorname{gra}}
\providecommand{\Id}{\operatorname{{ Id}}}

\providecommand{\fady}{\varnothing}

\providecommand{\argmin}{\mathrm{arg}\!\min}

\providecommand{\rras}{\rightrightarrows}
\providecommand{\To}{\rightrightarrows}

\providecommand{\NN}{\mathbb{N}}

\providecommand{\fix}{\operatorname{Fix}}
\providecommand{\ran}{\operatorname{ran}}

\providecommand{\Id}{\operatorname{Id}}

\providecommand{\zer}{\operatorname{zer}}

\providecommand{\fady}{\varnothing}

\providecommand{\bhmon}{{$\rho$-comonotone}}
\providecommand{\bhmaxmon}{{maximally $\rho$-comonotone}}
\providecommand{\rhmon}{{$\rho$-comonotone}}
\providecommand{\rhmaxmon}{{maximally $\rho$-comonotone}}
\providecommand{\bmon}{{$\rho$-comonotone}}
\providecommand{\bmaxmon}{{maximally $\rho$-monotone}}

\providecommand{\RR}{\mathbb{R}}
\providecommand{\NN}{\mathbb{N}}

\definecolor{dgreen}{rgb}{0.00,0.49,0.00}
\definecolor{dblue}{rgb}{0,0.08,0.55}
\colorlet{minhblue}{dblue}
\colorlet{mygreen}{dgreen}
\usepackage{hyperref}
\hypersetup{
    linktocpage=true,
    colorlinks= true,
    urlcolor = Red,
    citecolor = mygreen,
    linkcolor = dblue
}

\usepackage[colorinlistoftodos,prependcaption]{todonotes}

\definecolor{myblue}{rgb}{.8, .8, 1}
  \newcommand*\mybluebox[1]{%
    \colorbox{myblue}{\hspace{1em}#1\hspace{1em}}}

\allowdisplaybreaks 

\begin{document}

%

\author{
Heinz H.\ Bauschke\thanks{
Mathematics, University
of British Columbia,
Kelowna, B.C.\ V1V~1V7, Canada. E-mail:
\texttt{heinz.bauschke@ubc.ca}.},
~ Walaa M.\ Moursi\thanks{
Department of Electrical Engineering,
Stanford University,
350 Serra Mall, Stanford, CA 94305,
USA
and
Mansoura University, Faculty of Science,
Mathematics Department,
Mansoura 35516, Egypt.
E-mail: \texttt{wmoursi@stanford.edu}.}
~and~ Xianfu Wang\thanks{
Mathematics, University of
British Columbia,
Kelowna, B.C.\ V1V~1V7, Canada. E-mail:
\texttt{shawn.wang@ubc.ca}.
}}

\title{\textsc
Generalized monotone operators\\
and their averaged resolvents}

\date{February 22, 2019}

\maketitle

\begin{abstract}
\noindent

The correspondence between the monotonicity
of a (possibly) set-valued operator and the firm
nonexpansiveness of its resolvent is
a key ingredient in the convergence analysis
of many optimization algorithms.
Firmly nonexpansive operators form a proper
subclass of the  more general --
but still pleasant from an algorithmic perspective --
class of averaged operators.
In this paper, we introduce the new notion of conically nonexpansive operators
which generalize nonexpansive mappings. 
We characterize averaged operators
as being resolvents of comonotone
operators under appropriate scaling.
As a consequence, we characterize the proximal point mappings associated with 
hypoconvex functions as cocoercive operators, or equivalently;
as displacement mappings of conically nonexpansive operators.
Several examples illustrate our analysis
and demonstrate tightness of our results.

\end{abstract}
{\small

}

\section{Introduction}

In this paper, we assume that
\begin{empheq}[box=\mybluebox]{equation*}
\label{T:assmp}
\text{$X$ is a real Hilbert space},
\end{empheq}
with inner product $\innp{\cdot,\cdot}$ and
induced norm $\norm{\cdot}$.
Monotone operators form a beautiful class
of operators that play a crucial role in
modern optimization.
This class includes subdifferential operators
of proper lower semicontinuous convex functions
as well as
matrices with positive semidefinite symmetric part.
(For detailed discussions on monotone operators
and the connection to optimization problems,
we refer the reader to
\cite{BC2017},
\cite{Borwein50},
\cite{Brezis},
\cite{BurIus},
\cite{Comb96},
\cite{Comb04},
\cite{Mord18},
\cite{Rock98},
\cite{Simons1},
\cite{Simons2},
\cite{Zeidler2a},
\cite{Zeidler2b}, and the references therein.)

The correspondence
between the maximal monotonicity
 of an operator
and the firm nonexpansiveness
of its \emph{resolvent} is of central importance
from an algorithmic perspective:
to find a critical point of the former, iterate the later!

Indeed, firmly nonexpansive operators
belong to the more general and pleasant class
of \emph{averaged} operators.
Let $x_0\in X$ and let $T\colon X\to X$
be averaged.
Thanks to the Krasnosel'ski\u{\i}--Mann
iteration (see \cite{krans}, \cite{Mann} and
also \cite[Theorem~5.14]{BC2017}),
the sequence $(T^n x_0)_{n\in \NN}$
converges weakly to a fixed point of $T$.
When $T$ is the \emph{proximal mapping}
associated with a proper lower
semicontinuous convex function $f$,
the set of
fixed points of $T$ 
is the set of
critical point of $f$; equivalently
the set of minimizers of $f$.
In fact, iterating $T$ is this case
produces the famous proximal point algorithm,
 see \cite{Rock76}.
\emph{The main goal of this paper is to answer the question:
Can we explore a new correspondence
between a set-valued operator and its resolvent
which generalizes the fundamental
 correspondence between
 monotone operators and
 firmly nonexpansive mappings  (see
 \cref{fact:corres})?
Our approach relies on the new notion of
\emph{\conic\ }operators as well as the
notions of $\rho$-monotonicity
(respectively $\rho$-comonotonicity)
which, depending on the value of $\rho$,
reduce to strong monotonicity, monotonicity or hypomonotonicity
(respectively cocoercivity, monotonicity or cohypomonotonicity).}

Although some correspondences between a monotone operator 
$(\rho\geq 0)$ and its resolvent
have been
established in \cite{BMW2}, our analysis here 
not only provides more quantifications and
but also goes beyond monotone operators.
We now
summarize the three main results
of this paper:

\begin{itemize}
\item[{\bf R1}]
\namedlabel{R:1}{\bf R1}
We show that, when $\rho>-1$,
the resolvent of a
$\rho$-monotone operator
as well as the resolvent of its
inverse
are  single-valued and have full domain.
This allows us to extend the classical
theorem by Minty
(see \cref{thm:minty}) to this class of operators
(see \cref{thm:Minty:type}).

\item[{\bf R2}]
\namedlabel{R:2}{\bf R2}
We characterize \conic\ operators
(respectively averaged operators and nonexpansive operators)
to be resolvents of
$\rho$-comonotone operators
with $\rho>-1$
(respectively $\rho>-\tfrac{1}{2}$ and $\rho\ge -\tfrac{1}{2}$)
(see \cref{cor:eq:nexp:-0.3} and also \cref{tab:1}).
\item[{\bf R3}]
\namedlabel{R:3}{\bf R3}
As a consequence of \ref{R:2},
 we obtain a novel characterization
 of the proximal point mapping associated
 with a \emph{hypoconvex} function\footnote{This is also
 known as \emph{weakly convex} function.}
  (under
 appropriate scaling of the function)
 to be
 a \conic\ mapping, or equivalently,
 the displacement mapping
 of a cocoercive
 operator (see \cref{prop:hypocon:av}).
\end{itemize}

The remainder of this paper is organized as
follows.
 \cref{sec:mono:comono}
is devoted to the study of the
 properties of $\rho$-monotone and
 $\rho$-comonotone operators.
 In \cref{sec:aver}, we provide a characterization
 of averaged operators as resolvents of
   $\rho$-comonotone operators.
 \cref{sec:JA:RA} provides
 useful correspondences between
 an operators and its resolvent
 as well as its reflected resolvent.
In \cref{sec:linear}, we focus on
$\rho$-monotone and
 $\rho$-comonotone linear operators.
In the final \cref{sec:hypocon},
 we establish the connection to hypoconvex functions.

The notation we use is standard and
follows, e.g.,
 \cite{BC2017} or \cite{Rock98}.

\section{$\rho$-monotone and $\rho$-comonotone operators}
\label{sec:mono:comono}

Let $A\colon X\rras X$. Recall that the
\emph{resolvent} of $A$
is
$J_A=(\Id+A)^{-1} $ and the
\emph{reflected resolvent} of $A$
is
$R_A=2J_A-\Id $,
where $\Id\colon X\to X\colon x\mapsto  x$.
The \emph{graph} of $A$
is $\gra A=\menge{(x,u)\in X\times X}{u\in Ax}$.
Let $T\colon X\to X$
and let $\alpha \in\left]0,1\right[$.
Recall that
\begin{enumerate}
  \item
  $T$ is \emph{nonexpansive}
  if $(\forall (x,y)\in X\times X)$
  $\norm{Tx-Ty}\le \norm{x-y}$.
\item
$T $ is \emph{$\alpha$-averaged}
if there exists a nonexpansive
operator $N\colon X\to X$
such that $T=(1-\alpha)\Id+\alpha N$;
equivalently, $(\forall (x,y)\in X\times X)$
we have
\begin{equation}
(1-\alpha)\norm{(\Id-T)x-(\Id-Ty)}^2
\le \alpha(\norm{x-y}^2-\norm{Tx-Ty}^2).
\end{equation}
\item
$T $ is \emph{firmly nonexpansive}
if $T$ is $\tfrac{1}{2}$-averaged.
Equivalently, if $(\forall (x,y)\in X\times X)$
$\normsq{Tx-Ty}+\normsq{(\Id-T)x-(\Id-T)y}\le \normsq{x-y}$.
\end{enumerate}

We begin this section by
stating the following two useful facts.
\begin{fact}{\rm (see, e.g.,
  \cite[Theorem~2]{EckBer})}
  \label{fact:corres}
Let $D$ be a nonempty subset of
$X$, let $T\colon D\to X$,
and
set $A=T^{-1}-\Id$. Then $T=J_A$.
Moreover, the following hold:
\begin{enumerate}
    \item
    $T$ is firmly nonexpansive if and only if
    $A$ is monotone.
    \item
    $T$ is firmly nonexpansive and $D=X$
    if and only if
    $A$ is maximally monotone.
\end{enumerate}
\end{fact}

\begin{fact}[{\bf Minty's Theorem}]{\rm\cite{Minty}
  (see also \cite[Theorem~21.1]{BC2017})}
\label{thm:minty}
Let $A\colon X\rras X$ be monotone. Then
\begin{equation}
\label{eq:Minty}
\gra A=\menge{(J_A x, (\Id-J_A)x)}{x\in \ran (\Id+A)}.
\end{equation}
Moreover,
\begin{equation}
\label{eq:Minty:2}
\text{$A$ is maximally monotone $\siff$ $\ran (\Id+A)=X$.}
\end{equation}
\end{fact}

\begin{defn}
Let $A\colon X\rras X$ and let
$\rho\in \RR$.
Then
\begin{enumerate}
\item
$A$ is
\emph{$\rho$-monotone}
 if $(\forall (x,u)\in \gra A)$
$(\forall (y,v)\in \gra A)$ we have
\begin{equation}
\label{eq:def:hmon}
\innp{x-y,u-v}\ge \rho\normsq{x-y}.
\end{equation}
\item
$A$ is
\emph{maximally $\rho$-monotone}
if $A$ is $\rho$-monotone
and there is no $\rho$-monotone
operator $B\colon X\rras X$
such that $\gra B$ properly contains $\gra A$,
i.e., for every $(x,u)\in X\times X$,
\begin{equation}
(x,u)\in \gra A ~\siff~ (\forall (y,v)\in \gra A)
~\innp{x-y,u-v}\ge \rho\normsq{x-y}.
\end{equation}
\item
$A$ is
\emph{\bhmon}
if $(\forall (x,u)\in \gra A)$
$(\forall (y,v)\in \gra A)$ we have
\begin{equation}
\label{eq:def:cohmon}
\innp{x-y,u-v}\ge \rho\normsq{u-v}.
\end{equation}
\item
$A$ is
\emph{\bhmaxmon} if $A$ is \bhmon\
and there is no \bhmon\ operator $B\colon X\rras X$
such that $\gra B$ properly contains $\gra A$,
i.e., for every $(x,u)\in X\times X$,
\begin{equation}
(x,u)\in \gra A ~\siff~ (\forall (y,v)\in \gra A)
~\innp{x-y,u-v}\ge \rho\normsq{u-v}.
\end{equation}
\end{enumerate}
\end{defn}
Some comments are in order.
\begin{rem}\
\begin{enumerate}
\item
When $\rho=0$,
both $\rho$-monotonicity of $A$
 and $\rho$-comonotonicity of $A$
 reduce to the monotonicity
 of $A$; equivalently
 to the monotonicity of $A^{-1}$.

\item
When $\rho< 0$, $\rho$-monotonicity
is known as $\rho$-\emph{hypomonotonicity},
see \cite[Example~12.28]{Rock98}
and
\cite[Definition~6.9.1]{BurIus}.
In this case,
the $\rho$-comonotonicity
is also known as
\emph{$\rho$-cohypomonotonicity}
(see \cite[Definition~2.2]{CombPenn04}).

\item
In passing, we point out that when
$\rho>0$, $\rho$-monotonicity of $A$
reduces to $\rho$-strong monotonicity of $A$,
while  $\rho$-comonotonicity of $A$
 reduces to $\rho$-cocoercivity\footnote{Let $\beta>0$
and let $T\colon X\to X$.
Recall
that $T$ is $\beta$-\emph{cocoercive}
if $\beta T$ is firmly nonexpansive,
i.e., $(\forall (x,y)\in X\times X)$
$\innp{x-y,Tx-Ty}\ge
\beta \normsq{Tx-Ty}$.
}
of $A$.

\end{enumerate}
\end{rem}

Unlike classical monotonicity,
$\rho$-comonotonicity of $A$ is \emph{not} equivalent to
 $\rho$-comonotonicity of $A^{-1}$.
Instead,  we have the following correspondences.

\begin{lemma}
\label{lem:bmon:inv}
Let $A\colon X\rras X$ and let $\rho\in \RR$.
The following are equivalent:
\begin{enumerate}
\item
\label{lem:bmon:inv:i}
$A$ is \rhmon.
\item
\label{lem:bmon:inv:ii}
$A^{-1}-\rho\Id$ is monotone.
\item
\label{lem:bmon:inv:iii}
$A^{-1}$ is $\rho$-monotone, i.e.,
$(\forall (x,u)\in \gra A^{-1})$ $(\forall (y,v)\in \gra A^{-1})$
$
\innp{x-y,u-v}\ge \rho \normsq{x-y}.
$
\end{enumerate}
\end{lemma}
\begin{proof}
``\ref{lem:bmon:inv:i}$\RA$\ref{lem:bmon:inv:ii}":
Let $\{(x,u),(y,v) \}\subseteq X\times X$.
Then $\{(x,u),(y,v) \}\subseteq\gra (A^{-1}-\rho\Id)$
$\siff$ [$u\in A^{-1}x-\rho x $ and $v\in A^{-1}y-\rho y $]
$\siff$ $\{(x, u+\rho x),(y, v+\rho y)\}\subseteq \gra A^{-1}$
$\siff$ $\{( u+\rho x, x),(v+\rho y,y)\}\subseteq \gra A$
$\RA$ $\innp{x-y,u-v+\rho(x-y)}\ge \rho\normsq{x-y}$
$\siff$ $\rho\normsq{x-y}+\innp{x-y,u-v}\ge \rho\normsq{x-y}$
$\siff$ $\innp{u-v,x-y}\ge 0$.

``\ref{lem:bmon:inv:ii}$\RA$\ref{lem:bmon:inv:iii}":
Let $\{(x,u),(y,v) \}\subseteq\gra A^{-1}$.
Then $\{(x,u-\rho x),(y,v-\rho y) \}\subseteq\gra (A^{-1}-\rho \Id)$.
Hence $\innp{x-y, u-v-\rho(x-y)}\ge 0$;
equivalently $\innp{x-y, u-v}\ge\rho\normsq{x-y}$.

``\ref{lem:bmon:inv:iii}$\RA$\ref{lem:bmon:inv:i}":
Let $\{(x,u),(y,v) \}\subseteq X\times X$.
Then $\{(x,u),(y,v) \}\subseteq\gra A$
$\siff$ $\{(u,x),(v,y) \}\subseteq \gra A^{-1}$
$\RA $ $\innp{x-y,u-v}\ge \rho\normsq{u-v}$.
\end{proof}

\begin{lemma}
\label{lem:gra:conn}
Let $A\colon X\rras X$ and let $\rho\in \RR$.
Then the following hold:
\begin{enumerate}
\item
\label{eq:grA:grainv}
$\gra A=\menge{(u+\rho x,x)}{(x,u)\in \gra (A^{-1}-\rho \Id)}$.
\item
\label{eq:grainv:grA}
$\gra (A^{-1}-\rho \Id)=\menge{(u,x-\rho u)}{(x,u)\in \gra A}$.
\end{enumerate}
\end{lemma}
\begin{proof}
\ref{eq:grA:grainv}:
Let $(x,u)\in X\times X$.
Then $(x,u)\in \gra (A^{-1}-\rho \Id)$
$\siff$ $u\in A^{-1}x-\rho x$
$\siff$ $u+\rho x\in A^{-1} x$
$\siff$ $x\in A(u+\rho x)$
$\siff$ $(u+\rho x,x)\in \gra A$.
This proves ``$\supseteq$" in \ref{eq:grA:grainv}.
The opposite inclusion can be proved similarly.
\ref{eq:grainv:grA}:
The proof proceeds similar to that of
\ref{eq:grA:grainv}.
\end{proof}

\begin{lemma}
\label{lem:bmax:inv}
Let $A\colon X\rras X$ and let $\rho\in \RR$.
The following are equivalent:
\begin{enumerate}
\item
\label{lem:bmax:inv:i}
$A$ is \rhmaxmon.
\item
\label{lem:bmax:inv:ii}
$A^{-1}-\rho\Id$ is maximally monotone.
\end{enumerate}
\end{lemma}
\begin{proof}
Note that \cref{lem:bmon:inv}
implies that
$A$ is \rhmon\ $\siff $
$A^{-1}-\rho\Id$ is monotone.
``\ref{lem:bmax:inv:i}$\RA$\ref{lem:bmax:inv:ii}":
Let $(y,v)\in X\times X$.
Then $(y,v)$ is monotonically
related to $\gra(A^{-1}-\rho\Id)$
$\siff$ $(\forall (x,u)\in \gra (A^{-1}-\rho \Id))$
$\innp{x-y,u-v}\ge 0$
$\siff$ $(\forall (x,u)\in \gra (A^{-1}-\rho \Id))$
$\innp{x-y,u-v}+\rho\normsq{x-y}\ge \rho\normsq{x-y}$
$\siff$ $(\forall (x,u)\in \gra (A^{-1}-\rho \Id))$
$\innp{x-y,u+\rho x-(v+\rho y)}\ge\rho\normsq{x-y}$.
Because the last inequality holds
for all $(x,u)\in \gra (A^{-1}-\rho \Id)$,
the
parametrization of $\gra A$ given in
\cref{lem:gra:conn}\ref{eq:grA:grainv} and
 the \emph{maximal} $\rho$-comonotonicity of $A$
 imply
 that
 $(v+\rho y, y) \in \gra A$.
Therefore, by \cref{lem:gra:conn}\ref{eq:grainv:grA},
 $(y,v)\in \gra (A^{-1}-\rho \Id)$.

 ``\ref{lem:bmax:inv:ii}$\RA$\ref{lem:bmax:inv:i}":
Let $(y,v)\in X\times X$.
Then $(y,v)$ is $\rho$-comonotonically related to $\gra A$
$\siff$ $(\forall (x,u)\in \gra A)$
$\innp{x-y,u-v}\ge \rho \normsq{u-v}$
$\siff$ $(\forall (x,u)\in \gra A)$
$\innp{x-\rho u-(y-\rho v),u-v}\ge 0$.
It follows from \cref{lem:gra:conn}\ref{eq:grainv:grA}
and 
 the \emph{maximal} monotonicity of $A^{-1}-\rho\Id$
that
$(v, y-\rho v)\in \gra (A^{-1}-\rho \Id)$,
equivalently, using \cref{lem:gra:conn}\ref{eq:grA:grainv},
$(y,v)\in \gra A$.
\end{proof}

\begin{rem}
Note that when $\rho<0$,
the (maximal) monotonicity of
$A^{-1}-\rho\Id$
is equivalent to
the (maximal) monotonicity of
the Yosida approximation
$(A^{-1}-\rho\Id)^{-1}$.
Such a characterization
is presented in \cite[Proposition~6.9.3]{BurIus}.
\end{rem}

\begin{prop}
\label{prop:surject}
Let $A\colon X\rras X$ be \bhmaxmon\ where $\rho>-1$.
Then
$\ran(\Id+A^{-1})=X$.

\end{prop}
\begin{proof}
By \cref{lem:bmax:inv}, $A^{-1}-\rho \Id$
 is maximally monotone.
 Consequently, because $1+\rho>0$,
 the operator $\tfrac{1}{1+\rho}(A^{-1}-\rho\Id)$
 is maximally monotone.
Applying \cref{eq:Minty:2} to
$\tfrac{1}{1+\rho}(A^{-1}-\rho\Id)$
we have $\ran (\Id+A^{-1})=\ran ((1+\rho)\Id+(A^{-1}-\rho \Id))
=(1+\rho)\ran (\Id+\tfrac{1}{1+\rho}(A^{-1}-\rho\Id))=(1+\rho)X=X$.
\end{proof}

\begin{prop}
\label{prop:gen:min}
Let $A\colon X\rras X$. 
Then the following hold:
\begin{enumerate}
\item
\label{prop:gen:min:i}
$J_{A^{-1}}=\Id-J_A.$
\item
\label{prop:gen:min:ii}
$\ran(\Id+A^{-1})=\dom(\Id-J_A)=\ran (\Id+A)$.
\end{enumerate}
\end{prop}
\begin{proof}
\ref{prop:gen:min:i}:
This follows from \cite[Proposition~23.7(ii)~and~Definition~23.1]{BC2017}.
\ref{prop:gen:min:ii}:
Using \ref{prop:gen:min:i}, we have
$\ran(\Id+A^{-1})
=\dom (\Id+A^{-1})^{-1}
=\dom J_{A^{-1}}
=\dom (\Id-J_A)
=(\dom \Id)\cap (\dom J_A)
=\dom J_A
=\ran (\Id+A)$.
\end{proof}

\begin{corollary}[\bf{surjectivity of $\Id+A$ and $\Id+A^{-1}$}]
\label{cor:surj}
Let $A\colon X\rras X$ be \bhmaxmon\ where $\rho>-1$.
Then
\begin{equation}
\dom J_A=\ran(\Id+A)=X,
\end{equation}
and
\begin{equation}
  \dom (\Id-J_A)=\ran(\Id+A^{-1})=X.
  \end{equation}
\end{corollary}
\begin{proof}
Combine
\cref{prop:surject}
 and
\cref{prop:gen:min}\ref{prop:gen:min:i}\&\ref{prop:gen:min:ii}.
\end{proof}

\begin{prop}[\bf{single-valuedness of the resolvent}]
\label{prop:s:v}
Let $A\colon X\rras X$ be \bhmon\ where $\rho>-1$.
Then $J_A=(\Id+A)^{-1}$ and
$J_{A^{-1}}=\Id-J_A$ are at most single-valued.
\end{prop}
\begin{proof}
Let $x\in \dom J_A=\ran(\Id+A)$ and
let $(u,v)\in X\times X$.
Then $\{u,v\}\subseteq J_A x$
$\siff$ [$x-u\in Au$ and $x-v\in Av$]
$\RA$ $\innp{(x-u)-(x-v),u-v}\ge \rho\normsq{u-v}$
$\siff$ $-\normsq{u-v}\ge \rho\normsq{u-v}$.
Since $\rho>-1$, the last inequality
implies that $u=v$.
Now combine with \cref{prop:gen:min}\ref{prop:gen:min:i}.
\end{proof}

\begin{corollary}[{See also \cite[Proposition~3.4]{PhanDao18}}]
  Let $A\colon X\rras X$ be \bhmaxmon\ where $\rho>-1$.
  Then $J_A=(\Id+A)^{-1}$ and
  $J_{A^{-1}}=\Id-J_A$ are single-valued
  and $\dom J_A=\dom J_{A^{-1}}=X$.
\end{corollary}

In \cref{ex:sv:fr:fail} below, we illustrate that
the assumption that $\rho >-1$
is critical in the conclusion of
\cref{cor:surj}
and
\cref{prop:s:v}.

\begin{example}
\label{ex:sv:fr:fail}
Suppose that $X\neq \{0\}$.
Let $C$ be a nonempty closed convex subset
of $X$, let $r\in \RR_+$, set $B=-\Id-rP_C$,
set $A=B^{-1}$ and set $\rho=-(1+r)\le -1$.
Then the following hold:
\begin{enumerate}
\item
\label{ex:sv:fr:fail:i}
$B-\rho \Id $ is maximally monotone.
\item
\label{ex:sv:fr:fail:ii}
$A$ is \rhmaxmon.
\item
\label{ex:sv:fr:fail:iii}
$\ran (\Id+A)=\ran(\Id+A^{-1})=(\rho+1)C=-rC$.
\item
\label{ex:sv:fr:fail:iii:b}
$\Id+A$ is surjective $\siff$
$[C=X \text{~and~}r>0]$.
\item
\label{ex:sv:fr:fail:iv}
$J_A$ is at most single-valued $\siff$
$J_{A^{-1}}$ is at most single-valued $\siff$
$[C=X \text{~and~}r>0]$.
\end{enumerate}
\end{example}

\begin{proof}
\ref{ex:sv:fr:fail:i}:
Indeed, $B-\rho\Id=-\Id-rP_C+(1+r)\Id=r(\Id-P_C)$.
It follows from
\cite[Example~23.4~\&~Proposition~23.11(i)]{BC2017}
that $\Id-P_C$ is maximally monotone.
Because $r\ge 0$,
the operator $B-\rho\Id=r(\Id-P_C)$ is maximally monotone as well.

\ref{ex:sv:fr:fail:ii}:
Combine \ref{ex:sv:fr:fail:i} and \cref{lem:bmax:inv}.

\ref{ex:sv:fr:fail:iii}:
The first identity is \cref{prop:gen:min}\ref{prop:gen:min:ii}.
Now $\ran (\Id+A^{-1})=\ran(\Id+B)=\ran(-rP_C)
=-r\ran P_C=-rC=(\rho+1)C$.

\ref{ex:sv:fr:fail:iii:b}:
This is a direct consequence of \ref{ex:sv:fr:fail:iii}.

\ref{ex:sv:fr:fail:iv}:
The first equivalence follows from
\cref{prop:gen:min}\ref{prop:gen:min:i}.
Note that $[r=0 \text{~or~} C=\{0\}]$
$\siff rC=\{0\}$
$\siff rP_C\equiv 0$
$\siff$ $B=-\Id$
$\siff$ $\gra J_{A^{-1}}=\gra J_B=\{0\}\times X$.
Now suppose that $r> 0$. Then
 $J_{A^{-1}}=J_B=(\Id+B)^{-1}=(-rP_C)^{-1}
=(\Id+N_C)\circ(-r^{-1}\Id)$ which is
at most single-valued
$\siff $ $C=X$, by e.g., \cite[Theorem~7.4]{BC2017}.
\end{proof}

\begin{prop}
\label{prop:surj:ow}
Let $A\colon X\To X$ be \bhmon,
where $\rho>-1$, and
such that $\ran(\Id+A)=X$.
Then $A$ is \bhmaxmon.
\end{prop}
\begin{proof}
Let $(x,u)\in X\times X$
such that
$(\forall (y,v)\in \gra A)$
\begin{equation}
\label{eq:min:od}
\innp{x-y,u-v}\ge \rho\normsq{u-v}.
\end{equation}
It follows from the surjectivity of
$\Id+A$ that there exists $(y,v)\in X\times X$
such that $v\in Ay$ and $x+u=y+v\in (\Id +A)y$.
Consequently, \cref{eq:min:od} implies that
$\rho\normsq{u-v}\le \innp{x-y,u-v}
=\innp{-(u-v),u-v}=-\normsq{u-v}$.
Hence, because $\rho> -1$,
we have $u=v$ and thus $x=y$
which proves the maximality of $A$.
\end{proof}
\begin{thm}[{\bf Minty parametrization}]
\label{thm:Minty:type}
Let $A\colon X\rras X$ be \bhmon\ where $\rho>-1$. Then
\begin{equation}
\gra A=\menge{(J_A x, (\Id-J_A)x)}{x\in \ran (\Id+A)}.
\end{equation}
Moreover, $A$ is \bhmaxmon\ $\siff$ $\ran (\Id+A)=X$,
in which case
\begin{equation}
\gra A=\menge{(J_A x, (\Id-J_A)x)}{x\in X}.
\end{equation}
\end{thm}
\begin{proof}
Let $(x,u)\in X\times X$.
In view of \cref{prop:s:v}
we have $(x,u)\in \gra A$$\siff u\in Ax$
$\siff x+u\in x+Ax=(\Id +A)x$
$\siff x=J_A(x+u)$
$\siff$ [$z:=x+u\in \ran (\Id +A)$,
 $x=J_A z$ and $u=x+u-x=x+u-J_A(x+u)=(\Id-J_A)z$].
The equivalence of maximal $\rho$-comonotonicity of $A$
and the surjectivity of $\Id +A$ follows from
combining \cref{cor:surj} and \cref{prop:surj:ow}.
\end{proof}

\begin{corollary}
\label{lem:gr:RA}
Suppose that $A\colon X\rras X$
is maximally $\rho$-comonotone
where $\rho>-1$
 and let $(x,u)\in X\times X$.
Then the following hold:
\begin{enumerate}
\item
\label{lem:gr:RA:i}
$
(x,u)\in \gra J_A\siff (u, x-u) \in \gra A
$.
\item
\label{lem:gr:RA:ii}
$
(x,u)\in \gra R_A\siff \bigl(\tfrac{1}{2}(x+u),
\tfrac{1}{2}(x-u)
\bigr) \in \gra A  $.
\end{enumerate}
\end{corollary}
\begin{proof}
Let $(x,u)\in X\times X$
 and note that in view of \cref{prop:s:v}
 and \cref{thm:Minty:type}
 $J_A\colon X\to X$ and consequently $R_A\colon X\to X$
 are single-valued.

\ref{lem:gr:RA:i}:
We have $(x,u)\in \gra J_A$
$\siff$ $u=J_Ax$
$\siff$ $x-u=(\Id-J_A)x$.
Now use \cref{thm:Minty:type}.

\ref{lem:gr:RA:ii}:
We have $(x,u)
\in\gra R_A\siff u=R_A x=2J_Ax-x$
$\siff x+u=2J_A x$
$\siff J_A x=\tfrac{1}{2}(x+u)$
$\siff x-J_Ax=x-\tfrac{1}{2}(x+u)=\tfrac{1}{2}(x-u)$
$\siff (\tfrac{1}{2}(x+u), \tfrac{1}{2}(x-u)) \in \gra A$,
where the last equivalence follows from
\cref{thm:Minty:type}.
\end{proof}

\section{$\rho$-comonotonicity and averagedness}
\label{sec:aver}
We start this section
with the following definition.
\begin{definition}
\label{def:con:nonexp}
  Let $T\colon X\to X$
  and let $\alpha\in \left ]0,+\infty\right [$.
  Then $T $ is $\alpha$-\emph{\conic} if there
  exists a nonexpansive operator $N\colon X\to X$
  such that $T=(1-\alpha)\Id+\alpha N$.
\end{definition}

\begin{remark}
\label{rem:con:nonexp}
In view of \cref{def:con:nonexp},
it is clear that
$T$ is $\alpha$-averaged if and only if
[$T$ $\alpha$-\conic\ and $\alpha\in \left ]0,1\right [$].
Similarly,
$T$ is nonexpansive if and only if
$T$ $1$-\conic.
\end{remark}

The proofs of the next two results
are straightforward and hence omitted.

\begin{lemma}
  \label{lem:coco:conc}
  Let $T\colon X\to X$
  and let $\alpha\in \left ]0,+\infty\right [$.
  Then
  \begin{equation}
 \text{ $T $ is $\alpha$-\conic\
  $\siff$
  $\Id-T$ is $\tfrac{1}{2\alpha}$-cocoercive}.
  \end{equation}
\end{lemma}

\begin{lemma}
\label{lem:lip:cute}
Let $D$ be a nonempty subset of
$X$, let $T\colon D\to X$,
let $N \colon  D\to X$,
 let $\alpha\in\left[1,+\infty\right[$
and set $T=(1-\alpha)\Id+\alpha N$.
Suppose that  $N\colon D\to X$ is nonexpansive.
Then
$(\forall (x,y)\in D\times D)$ we have
\begin{equation}
\norm{Tx-Ty}\le (2\alpha-1)\norm{x-y},
\end{equation}
i.e., $T$ is Lipschitz with constant $2\alpha-1$.
\end{lemma}

One can directly
verify the following result.
\begin{lemma}
\label{lem:gen:Hilb}
Let $(x,y)\in X\times X$
and let $\alpha\in \RR$.
Then
\begin{equation}
\alpha^2\normsq{x}-\normsq{(\alpha-1)x+y}
=2\alpha\innp{x-y,y}-(1-2\alpha)\normsq{x-y}.
\end{equation}
\end{lemma}

\begin{lemma}
\label{lem:av:ch}
Let $D$ be a nonempty subset of
$X$,
let $N \colon  D\to X$,
 let $\alpha\in \RR$
and set $T=(1-\alpha)\Id+\alpha N$.
Then $N$ is nonexpansive
if and only if
$(\forall (x,y)\in D\times D)$ we have
\begin{equation}
2\alpha\innp{Tx-Ty,(\Id-T)x-(\Id-T)y}\ge (1-2\alpha)\normsq{(\Id-T)x-(\Id-T)y}.
\end{equation}
\end{lemma}

\begin{proof}
  Let $(x,y)\in D\times D$.
Applying \cref{lem:gen:Hilb} with $(x,y)$ replaced by
$(x-y, Tx-Ty)$, we learn that
\begin{subequations}
\begin{align}
&\qquad2\alpha\innp{Tx-Ty,(\Id-T)x-(\Id-T)y}- (1-2\alpha)\normsq{(\Id-T)x-(\Id-T)y}\\
&=\alpha^2\normsq{x-y}-\normsq{(\alpha-1)(x-y)+(1-\alpha)(x-y)+\alpha(Nx-Ny)}\\
&=\alpha^2\big(\normsq{x-y}-\normsq{Nx-Ny}\big).
\end{align}
\end{subequations}
Now $N$ is nonexpansive $\siff$ $\normsq{x-y}-\normsq{Nx-Ny}\ge 0$
and the conclusion directly follows.
\end{proof}

We now provide new characterizations
of averaged and nonexpansive
operators.
\begin{corollary}
\label{cor:non:av:ch}
Let $D$ be a nonempty subset of
$X$, let $T\colon D\to X$,
let $\alpha\in\left]0, +\infty\right[$
and let $(x,y)\in D\times D$.
 Then the following hold:
 \begin{enumerate}
 \item
 \label{cor:non:av:ch:i}
$T$ is nonexpansive
$\siff $
$2\innp{Tx-Ty,(\Id-T)x-(\Id-T)y}\ge -\normsq{(\Id-T)x-(\Id-T)y}$.
\item
\label{cor:non:av:ch:ii}
$T$ is $\alpha$-\conic\
$\siff$
 $2\alpha\innp{Tx-Ty,(\Id-T)x-(\Id-T)y}\ge (1-2\alpha)\normsq{(\Id-T)x-(\Id-T)y}$.
\end{enumerate}
\end{corollary}
\begin{proof}
\ref{cor:non:av:ch:i}:
Apply \cref{lem:av:ch} with $\alpha=1$.

\ref{cor:non:av:ch:ii}:
A direct consequence of
\cref{lem:av:ch}.
\end{proof}

\begin{prop}
\label{prop:gen:avn}
Let $D$ be a nonempty subset of $X$,
let $T\colon D\to X$, let $\alpha\in \left]0,+\infty\right[$,
set $A=T^{-1}-\Id$
 and set $N=\tfrac{1}{\alpha}T-\tfrac{1-\alpha}{\alpha}\Id$,
 i.e.,
 $T=J_A=(1-\alpha)\Id+\alpha N$.
Then the following hold:
\begin{enumerate}
\item
\label{prop:gen:avn:iii}
$T$ is $\alpha$-\conic\
$\siff$
$N$ is nonexpansive $\siff$
$A$ is $\big(\tfrac{1}{2\alpha}-1\big)$-comonotone.
\item
\label{prop:gen:avn:iv}
{\rm [}$T$ is $\alpha$-\conic\ and $D=X${\rm ]}
$\siff$
{\rm [}$N$ is nonexpansive and $D=X${\rm ] }
$\siff$ $A$ is
maximally
$\big(\tfrac{1}{2\alpha}-1\big)$-comonotone.
\end{enumerate}

\end{prop}
\begin{proof}
\ref{prop:gen:avn:iii}:
The first equivalence is \cref{def:con:nonexp}.
We now turn to the second equivalence.
``$\RA$": Let $\{(x,u),(y,v)\}\subseteq \gra A$.
Then
$(x,u)=(T(x+u),(\Id-T)(x+u))$
and likewise $(y,v)=(T(y+v),(\Id-T)(y+v))$.
It follows from \cref{lem:av:ch} applied with
$(x,y)$ replaced by $(x+u,y+v)$
that
$2\alpha\innp{x-y,u-v}\ge (1-2\alpha)\normsq{u-v}$.
Since $\alpha>0$, the conclusion
follows by dividing both sides of the last inequality by
$2\alpha$.
``$\LA$":
Using \cref{thm:Minty:type},
we learn that $(\forall(x,y)\in D\times D)$
$\{(Tx,(\Id-T)x),(Ty,(\Id-T)y)\}\subseteq\gra A$
and hence
$\innp{Tx-Ty,(\Id-T)x-(\Id-T)y}\ge \bk{\tfrac{1}{2\alpha}-1}\normsq{(\Id-T)x-(\Id-T)y}$.
Thus
$2\alpha\innp{Tx-Ty,(\Id-T)x-(\Id-T)y}\ge(1-2\alpha)\normsq{(\Id-T)x-(\Id-T)y}$.
Now use \cref{lem:av:ch}.

\ref{prop:gen:avn:iv}:
Note that $\dom N=\dom T=\ran T^{-1}=\ran(\Id+A)$.
Now combine \ref{prop:gen:avn:iii}
 and \cref{thm:Minty:type}.
\end{proof}

\begin{prop}
\label{prop:-0.3:av}
Let $D$ be a nonempty subset of $X$,
let $T\colon D\to X$,
let $\alpha\in \left]0,+\infty\right[$,
set $A=T^{-1}-\Id$,
i.e., $T=J_A$,
and set $\rho=\tfrac{1}{2\alpha}-1>-1$.
Then the following equivalences hold:
\begin{enumerate}
\item
\label{prop:-0.3:conic:ii}
$T$ is $\alpha$-\conic\
$\siff$
$A$ is $\rho$-comonotone.
\item
\label{prop:-0.3:conic:iii}
{\rm [}$T$ is $\alpha$-\conic\ and $D=X$ {\rm ]}
$\siff$ $A$ is
maximally $\rho$-comonotone.
\item
\label{prop:-0.3:nonexp:ii}
$T$ is nonexpansive
$\siff$
$A$ is $\bigl(-\tfrac{1}{2}\bigr)$-comonotone.
\item
\label{prop:-0.3:nonexp:iii}
{\rm [}$T$ is nonexpansive and $D=X$ {\rm ]}
$\siff$ $A$ is
maximally $\bigl(-\tfrac{1}{2}\bigr)$-comonotone.
\end{enumerate}
If we  assume that
$\alpha\in \left]0,1\right[$, equivalently,
$\rho>-\tfrac{1}{2}$, then we additionally
have:
\begin{enumerate}
  \setcounter{enumi}{4}
  \item
\label{prop:-0.3:av:ii}
$T$ is $\alpha$-averaged $\siff$ $A$ is \bmon.
\item
\label{prop:-0.3:av:iii}
{\rm [}$T$ is $\alpha$-averaged and $D=X${\rm ]}
$\siff$ $A$ is \bmaxmon.
\end{enumerate}

\end{prop}

\begin{proof}
\ref{prop:-0.3:conic:ii}\&\ref{prop:-0.3:conic:iii}:
This follows from
\cref{prop:gen:avn}\ref{prop:gen:avn:iii}\&\ref{prop:gen:avn:iv}.
\ref{prop:-0.3:nonexp:ii}--\ref{prop:-0.3:av:iii}:
Combine \ref{prop:-0.3:conic:ii}
 and
 \ref{prop:-0.3:conic:iii}
 with \cref{rem:con:nonexp}.
\end{proof}
\begin{corollary}{\rm\bf(The characterization corollary).}
\label{cor:eq:nexp:-0.3}
Let $T\colon X\to X$.
Then the following hold:
\begin{enumerate}
  \item
\label{cor:eq:nexp:-0.3:i}
$T$ is nonexpansive if and
only if it is the resolvent of a maximally
$\bigl(-\tfrac{1}{2}\bigr)$-comonotone
operator $A\colon X\rras X$.
  \item
\label{cor:eq:nexp:-0.3:0}
Let $\alpha\in \left]0,+\infty\right[$.
Then
$T$ is $\alpha$-\conic\ if and
only if it is the resolvent of a
\bmon\
operator $A\colon X\rras X$,
where $\rho=\tfrac{1}{2\alpha}-1>-1$
\big(i.e.,
$\alpha=\tfrac{1}{2(\rho+1)}$\big).

\item
\label{cor:eq:nexp:-0.3:ii}
Let $\alpha\in \left]0,1\right[$.
Then $T$ is $\alpha$-averaged if and
only if it is the resolvent of a
\bmon\
operator $A\colon X\rras X$
where $\rho=\tfrac{1}{2\alpha}-1>-\tfrac{1}{2}$
(i.e.,
$\alpha=\tfrac{1}{2(\rho+1)}$).
\end{enumerate}
\end{corollary}

\begin{example}
  Suppose that $U$ is a closed linear
  subspace of $X$ and  set $N=2P_U-\Id$.
  Let
  $\alpha\in \left[0,+\infty\right[$,
  set $T_\alpha=(1-\alpha)\Id+\alpha N$,
 and set
   $A_\alpha=(T_\alpha)^{-1}-\Id$.
   Then for every   $\alpha\in \left[0,+\infty\right[$,
   $T_\alpha$ is $\alpha$-conically nonexpansive and
   \begin{equation}
  A_\alpha
  =\begin{cases}
  N_U,&\text{if  }\alpha=\tfrac{1}{2};\\
  \tfrac{2\alpha}{1-2\alpha}P_{U^\perp}, &\text{otherwise}.
  \end{cases}
  \end{equation}
  Moreover, $A_\alpha$ is
  $\bigl(\tfrac{1}{2\alpha}-1\bigr)$-comonotone.
  \end{example}
  \begin{proof}
  First note that
  $T_\alpha=(1-\alpha)\Id+\alpha (2P_U-\Id)
  =(1-2\alpha)\Id+2\alpha P_U$.
  The case $\alpha=\tfrac{1}{2}$ is
  clear by, e.g., \cite[Example~23.4]{BC2017}.
  Now suppose that $\alpha\in \left[0,+\infty\right[ \smallsetminus \{\tfrac{1}{2}\}$,
  and let $y\in X$.
  Then $y\in A_\alpha x$
  $\siff x+y\in (\Id+A_\alpha) x$
  $\siff x=T_\alpha (x+y )=(1-2\alpha)
  (x+y)+2\alpha P_U(x+y)$
  $\siff x=x+y-2\alpha(\Id-P_U)(x+y)$
  $\siff y=2\alpha P_{U^\perp}(x+y)
  =2\alpha P_{U^\perp} x+2\alpha P_{U^\perp}y
  =2\alpha P_{U^\perp} x+2\alpha y$.
  Therefore,
  $y=\tfrac{2\alpha}{1-2\alpha} P_{U^\perp} x$,
   and the conclusion follows in view of
  \cref{cor:eq:nexp:-0.3}\ref{cor:eq:nexp:-0.3:0}.
  \end{proof}

\begin{prop}
\label{prop:av:-0.3}
Let $A\colon X\rras X$
be such that $\dom A\neq \fady$,
let $\rho\in  \left]-1,+\infty\right[$,
 set $D=\ran(\Id+A)$,
 set $T = {J_A}$,
 i.e.,
 $A=T^{-1}-\Id$,
and set $N= 2(\rho+1)T-(2\rho+1)\Id$,
i.e.,
$T=\tfrac{2\rho+1}{2(\rho+1)}\Id+\tfrac{1}{2(\rho+1)}N$.
Then the following equivalences hold:
\begin{enumerate}
\item
\label{prop:av:-0.3:ii}
$A$ is $\rho$-comonotone $\siff$ $N$ is nonexpansive.
\item
\label{prop:av:-0.3:iii}
$A$ is maximally $\rho$-comonotone
$\siff$ $N$ is nonexpansive and $D=X$.
\end{enumerate}

\end{prop}
\begin{proof}
\ref{prop:av:-0.3:ii}:
Set $\alpha=\tfrac{1}{2(\rho+1)}$
 and note that $\alpha>0$.
It follows from \cref{prop:s:v}
that $T=J_A$ is single-valued.
Now use \cref{prop:gen:avn}\ref{prop:gen:avn:iii}.
\ref{prop:av:-0.3:iii}:
Combine \ref{prop:av:-0.3:ii}
 and \cref{prop:gen:avn}\ref{prop:gen:avn:iv}.
\end{proof}

\begin{prop}
\label{prop:all:oth}
Let $A\colon X\rras X$
be such that $\dom A\neq \fady$,
let $\rho\in  \left]-1,+\infty\right[$,
 set $D=\ran(\Id+A)$,
  set $T={J_A}$,
  i.e.,
  $A=T^{-1}-\Id$,
   and set $\alpha=\tfrac{1}{2(\rho+1)}$.
Then we have the following equivalences:
\begin{enumerate}
\item
\label{prop:all:oth:i:i}
 $A$ is $\rho$-comonotone
$\siff$ $T$ is $\tfrac{1}{2(\rho+1)}$-\conic.
  \item
\label{prop:all:oth:i:ii}
 $A$ is maximally  $\rho$-comonotone
 $\siff$  $T$ is $\alpha$-\conic\ and $D=X$.
\item
\label{prop:all:oth:ii}
 $A$ is $\bigl(-\tfrac{1}{2}\bigr)$-comonotone
 $\siff$ $T$ is nonexpansive.
\item
\label{prop:all:oth:iii}
$A$ is maximally $\bigl(-\tfrac{1}{2}\bigr)$-comonotone
$\siff$  $T$ is nonexpansive and $D=X$.
\item
\label{prop:all:oth:iv}\
{\rm[}$A$ is \bmon\ and $\rho>-\tfrac{1}{2}${\rm]}
$\siff$ $T$ is $\alpha$-averaged.
\item
\label{prop:all:oth:v}\
{\rm [}$A$ is \bmaxmon\ and
$\rho>-\tfrac{1}{2}${\rm]} $\siff$ {\rm[}$T$ is
$\alpha$-averaged and $D=X${\rm]}.
\end{enumerate}
\end{prop}

\begin{proof}
\ref{prop:all:oth:i:i}--\ref{prop:all:oth:v}:
Use \cref{prop:-0.3:av}.
\end{proof}

\begin{corollary}
Let $A\colon X\rras X$ be maximally \bmon\
and $\rho>-\tfrac{1}{2}$.
Then $J_A$ is $\tfrac{1}{2(\rho+1)}$-averaged.
\end{corollary}

The following corollary provides an alternative proof
to \cite[Proposition~6.9.6]{BurIus}.

\begin{corollary}
\label{cor:zer:cl:con}
Let $A\colon X\rras X$ be maximally \bmon\
 and $\rho\ge -\tfrac{1}{2}$.
Then $\zer A$ is closed and convex.
\end{corollary}
\begin{proof}
It is clear that
$\zer A=\fix J_A$.
The conclusion now follows from
combining \cite[Corollary~4.14]{BC2017}
and \cref{prop:all:oth}\ref{prop:all:oth:iii}.
\end{proof}

Table~\ref{tab:1} below
 summarizes the main results of this section.

\begin{table}[H]
  \resizebox{0.78\textwidth}{!}{\begin{minipage}{\textwidth}
\begin{tabular}{p{3.5cm}|p{2.6cm}p{0.35cm}|p{2.2cm}
  p{0.35cm}|p{3.6cm}p{0.35cm}|p{3.5cm}}
  \toprule
 $\rho$& $A$ & &$A^{-1}$& &$J_A$
 & &  $J_{A^{-1}}$ \\
\hline
 \begin{tikzpicture}[baseline=-0.3ex]
 \draw[>=triangle 45, <->] (-2,0) -- (1.5,0);
 \draw [line width=0.75mm] (0,0)--(1.4,0);
 \foreach \x in {0}
             \draw (\x,1pt) -- (\x,-1pt)
             node[anchor=north]{\footnotesize\x};
 \node[circle,draw=black, fill=white,
 inner sep=0pt,minimum size=5pt] at (0,0) {};
 \end{tikzpicture}
 & $\rho$-cocoercive &$\siff$
 & $\rho$-strongly monotone & $\siff$
 & $\tfrac{1}{2(\rho+1)}$-\conic\
 &$\siff$
 & $(\rho+1)$-cocoercive
  \\
 \begin{tikzpicture}[baseline=-0.3ex]
 \draw[>=triangle 45, <->] (-2,0) -- (1.5,0);
 \draw [line width=0.75mm] (0,0)--(0,0);
 \foreach \x in {0}
             \draw (\x,1pt) -- (\x,-1pt)
             node[anchor=north] {\footnotesize\x};
 \node[circle,draw=black, fill=black,
 inner sep=0pt,minimum size=5pt] at (0,0) {};
 \end{tikzpicture}
 &monotone  &$\siff$
 & monotone &$\siff$
 &firmly nonexpansive&$\siff$
 &  firmly nonexpansive
 \\
 \begin{tikzpicture}[baseline=-0.3ex]
 \draw[>=triangle 45, <->] (-2,0) -- (1.5,0);
 \draw [line width=0.75mm] (-1/2,0)--(0,0);
 \foreach \x in {0,-0.5}
             \draw (\x,1pt) -- (\x,-1pt)
             node[anchor=north] {\footnotesize\x};
             \node[circle,draw=black, fill=white, inner sep=0pt,minimum size=5pt] at (0,0) {};
             \node[circle,draw=black, fill=white, inner sep=0pt,minimum size=5pt] at (-0.5,0) {};
 \end{tikzpicture}
  &$\rho$-comonotone &$\siff$
  & $\rho$-monotone &$\siff$
  &$\tfrac{1}{2(\rho+1)}$-averaged &$\siff$
  & $(\rho+1)$-cocoercive
  \\
           \begin{tikzpicture}[baseline=-0.3ex]
  \draw[>=triangle 45, <->] (-2,0) -- (1.5,0);
  \foreach \x in {-0.5}
  \draw (\x,1pt) -- (\x,-1pt)
              node[anchor=north] {\footnotesize \x};
  \node[circle,draw=black, fill=black,
  inner sep=0pt,minimum size=5pt] at (-0.5,0){};
  \end{tikzpicture}
  & $\rho$-comonotone
  &$\siff$
  & $\rho$-monotone
  &$\siff$
  & nonexpansive
  &$\siff$
  & $\tfrac{1}{2}$-cocoercive
    \\
 \begin{tikzpicture}[baseline=-0.3ex]
 \draw[>=triangle 45, <->] (-2,0) -- (1.5,0);
 \draw [line width=0.75mm] (-1,0)--(-0.5,0);
 \foreach \x in {-1,-0.5}
             \draw (\x,1pt) -- (\x,-1pt)
             node[anchor=north] {\footnotesize \x};
             \node[circle,draw=black, fill=white,
             inner sep=0pt,minimum size=5pt] at
             (-0.5,0) {};
             \node[circle,draw=black, fill=white,
             inner sep=0pt,minimum size=5pt] at
             (-1,0) {};
 \end{tikzpicture}
    & $\rho$-comonotone
   &$\siff$
   & $\rho$-monotone
   &$\siff$
   & $\tfrac{1}{2(\rho+1)}$-\conic\
   &$\siff$
   & $(\rho+1)$-cocoercive
   \\
  \begin{tikzpicture}[baseline=-0.3ex]
  \draw[>=triangle 45, <->] (-2,0) -- (1.5,0);
  \draw [line width=0.75mm] (-1.85,0)--(-1,0);
  \foreach \x in {-1}
  \draw (\x,1pt) -- (\x,-1pt)
              node[anchor=north] {\footnotesize \x};
  \node[circle,draw=black, fill=black,
  inner sep=0pt,minimum size=5pt] at (-1,0){};
  \end{tikzpicture}
   & $\rho$-comonotone
   &$\siff$
   & $\rho$-monotone
   &
   $\RA$
   & may fail to be at most single-valued
   &
   $\siff$
   & may fail to be at most single-valued
    \\
   \toprule
\end{tabular}
\end{minipage}}
\caption{Properties
of an operator $A$ and its inverse
$A^{-1} $ along with the corresponding resolvents $J_A$
and $J_{A^{-1}}$ respectively,
for different values of $\rho\in \RR$.
Here, $A$ satisfies the implication:
$\{(x,u),(y,v)\}\subseteq
\gra A \RA \innp{x-y,u-v}\ge
\rho\normsq{u-v}$.}
\label{tab:1}
\end{table}

\section{Further properties of
the resolvent $J_A$ and the reflected
resolvent $R_A$}
\label{sec:JA:RA}

We start this section with
the following useful lemma.

\begin{lemma}
\label{lem:JA:RA:corr}
Let $T\colon X\to X$, let $\alpha\in
\left[0,1\right[$. Then the following hold:
\begin{enumerate}
\item
\label{lem:JA:RA:corr:ii}
$T$ is $\alpha$-averaged
$\siff$ $ 2T-\Id=(1-2\alpha)\Id +2\alpha N$
 for some nonexpansive
 $N\colon X\to X$.
\item
\label{lem:JA:RA:corr:iv}
$[T=\tfrac{\alpha}{2}(\Id+N)$ and $N$ is nonexpansive$]$ $\siff $ $-(2T-\Id)$ is
$\alpha$-averaged\footnote{This is
also known as
$\alpha$-negatively averaged (see \cite[Definition~3.7]{Gisel17}).},
in which case $T$ is a Banach contraction
with Lipschitz constant $\alpha<1$.
\item
\label{lem:JA:RA:corr:iii}
$T$ is $\tfrac{1}{2}$-strongly monotone
$\siff $ $2T-\Id$ is monotone.
\end{enumerate}
\end{lemma}
\begin{proof}
\ref{lem:JA:RA:corr:ii}:
We have: $T$ is $\alpha$-averaged $\siff $
[$T=(1-\alpha)\Id+\alpha N$ and $N$ is nonexpansive]
$\siff$
[$2T-\Id=(2-2\alpha)\Id+2\alpha N-\Id=(1-2\alpha)\Id+2\alpha N$
and $N$ is nonexpansive].

 \ref{lem:JA:RA:corr:iv}:
Indeed,
$[T=\tfrac{\alpha}{2}(\Id+N)$ and $N$ is nonexpansive$]$
$\siff $ $2T-\Id=(\alpha-1)\Id+\alpha N=-((1-\alpha)\Id+\alpha(-N))$,
equivalently $2T-\Id$ is
$\alpha$-negatively averaged.

\ref{lem:JA:RA:corr:iii}:
We have:
$T$ is $\tfrac{1}{2}$-strongly monotone
$\siff$ $T-\tfrac{1}{2}\Id $ is monotone
$\siff$ $2T-\Id $ is monotone.
\end{proof}

Before we proceed, we recall the following
useful fact (see, e.g., \cite[Proposition~4.35]{BC2017}). 
\begin{fact}
\label{Proposition:4.35}
Let $T\colon X\to X$, let $(x,y)\in X\times X$
 and let $\alpha \in]0, 1[$.
 Then 
 \begin{equation}
\text{$T$ is $\alpha$- averaged $\siff$
$\normsq{Tx-Ty}+(1-2\alpha) \normsq{x-y}
 \le 2(1-\alpha)\innp{x-y,Tx-Ty}$.}
 \end{equation}
\end{fact}

\begin{prop}
\label{p:corres:A:RA}
Let $\alpha\in
\left]0,1\right[$, let $\beta\in
\bigl]-\tfrac{1}{2},+\infty\bigr[$,
let $A\colon X\rras X$
 and
 suppose that $A$ is $\beta$-comonotone.
Then the following hold:
\begin{enumerate}
\item
\label{p:corres:A:RA:ii}
$A$ is $\beta$-comonotone
$\siff$
$J_A$ is $\tfrac{1}{2(1+\beta)}$-averaged
$\siff$
$R_A=\big(1-\tfrac{1}{1+\beta}\big )\Id+\tfrac{1}{1+\beta}N$
for some nonexpansive $N\colon X\to X$.
\item
\label{p:corres:A:RA:iv}
$A$ is $\beta$-strongly monotone
$\siff$
$[J_A=\tfrac{1}{2(\beta+1)}(\Id+N)$ and $N$ is nonexpansive$]$
$\siff$
 $-R_A$ is
$\tfrac{1}{\beta+1}$-averaged,
in which case
$J_A$ is a Banach contraction
with Lipschitz constant $\tfrac{1}{\beta+1}<1$.
\item
\label{p:corres:A:RA:iii}
$A$ is nonexpansive
$\siff$
$J_A$ is $\tfrac{1}{2}$-strongly monotone
$\siff$
$R_A$ is monotone.
\item
\label{p:corres:A:RA:i}
$A$ is $\alpha$-averaged
$\siff$
$R_A$ is $\tfrac{1-\alpha}{\alpha}$-cocoercive.
\item
\label{p:corres:A:RA:i:d}
$A$ is firmly nonexpansive
$\siff$ $R_A$ is firmly nonexpansive.
\end{enumerate}
\end{prop}
\begin{proof}
Let
$\{(x,u),(y,v)\}\subseteq X \times X$.
Using \cref{lem:gr:RA}\ref{lem:gr:RA:i},
we have
$
\{(x,u),(y,v)\}\subseteq\gra J_A$
 $\siff \{(u, x-u),
 (v,y-v)\}\subseteq\gra A$,
 which we shall use repeatedly.

\ref{p:corres:A:RA:ii}:
Let
$\{(x,u),(y,v)\}\subseteq\gra J_A$.  We have
\begin{subequations}
   \begin{align}
   &~A \text{~is~} \beta\text{-comonotone~}
   \nonumber\\
   \siff&~\beta \normsq{(x-y)-(u-v)}\le \innp{(x-y)-(u-v), u-v}
  \\
   \siff&~\beta\normsq{x-y}+\beta\normsq{u-v}-2\beta\innp{x-y,u-v}
   \le \innp{x-y,u-v}-\normsq{u-v}
   \\
   \siff&~ \beta\normsq{x-y}+(\beta+1)\normsq{u-v}
   \le (2\beta+1) \innp{x-y,u-v}
   \\
   \siff & ~\normsq{u-v} +\tfrac{\beta}{\beta+1}\normsq{x-y}
   \le \tfrac{2\beta+1}{\beta+1}
   \innp{x-y,u-v}
   \\
      \siff &~ \normsq{u-v} +\big(1-\tfrac{1}{\beta+1}\big)\normsq{x-y}
   \le 2\big(1-\tfrac{1}{2(\beta+1)}\big)
   \innp{x-y,u-v}
   \\
   \siff &~ J_A \text{~is~} \tfrac{1}{2(\beta+1)}\text{-averaged}, \text{}
   \\
     \siff &~ \text{$R_A=\big(1-\tfrac{1}{1+\beta}\big )\Id+\tfrac{1}{1+\beta}N$
for some nonexpansive $N\colon X\to X$,}
   \end{align}
  \end{subequations}
   where the last two equivalences follow from
    \cref{Proposition:4.35}
     and
     \cref{lem:JA:RA:corr}\ref{lem:JA:RA:corr:ii},
     respectively.

       \ref{p:corres:A:RA:iv}:
  We start by proving the equivalence
  of the first and third statement.
(see \cite[Proposition~5.4]{Gisel17} for ``$\RA$" and
also \cite[Proposition~2.1(iii)]{MV18}).
  Let $\{(x,u),(y,v)\}\subseteq \gra (-R_A)$, i.e.,
 $\{(x,-u),(y,-v)\}\subseteq \gra R_A$.
 In view of \cref{lem:gr:RA}\ref{lem:gr:RA:ii},
 this is equivalent to
  $ \{(\tfrac{1}{2}(x-u), \tfrac{1}{2}(x+u)),
 (\tfrac{1}{2}(y-v), \tfrac{1}{2}(y+v))\}\subseteq\gra A$.
We have
\begin{subequations}
     \begin{align}
  & A \text{ is $\beta$-strongly monotone}
  \nonumber\\
    \siff~&
    \innp{(x-y)+(u-v),(x-y)-(u-v)}\ge \beta \normsq{(x-y)-(u-v)}
  \\
    \siff~&
   \normsq{x-y}-\normsq{u-v}\ge \beta\normsq{x-y}+\beta\normsq{u-v}
   -2\beta\innp{x-y,u-v}
  \\
           \siff~&
  2\beta\innp{x-y,u-v} \ge (\beta-1)\normsq{x-y}+(\beta+1)\normsq{u-v}
  \\
         \siff~&
    \tfrac{ 2\beta}{\beta+1}\innp{x-y,u-v}
    \ge \tfrac{\beta-1}{\beta+1}\normsq{x-y}+\normsq{u-v}
  \\
      \siff~&
    2\bk{1-\tfrac{1}{\beta+1}}\innp{x-y,u-v}
    \ge\bk{1-\tfrac{2}{\beta+1}}\normsq{x-y}+\normsq{u-v}
  \\
    \siff~& -R_A \text{~is~} \tfrac{1}{\beta+1}\text{-averaged},
     \end{align}
  \end{subequations}
   where the last equivalence follows from
    \cref{Proposition:4.35}.
    Now apply
     \cref{lem:JA:RA:corr}\ref{lem:JA:RA:corr:iv}
     to prove the equivalence of the second and third statements
     in \ref{p:corres:A:RA:iv}.

   \ref{p:corres:A:RA:iii}:
   Let
$\{(x,u),(y,v)\}\subseteq\gra J_A$
 and note that \cref{lem:gr:RA}\ref{lem:gr:RA:i}
 implies that
 $x-u\in Au$, $y-v\in Av$,
 $2u-x\in (\Id - A)u$
 and
 $2v-y\in (\Id - A)v$.
 It follows from
\cref{cor:non:av:ch}\ref{cor:non:av:ch:i}
applied with $(T,x,y)$ replaced by $(A,u,v)$
that
\begin{subequations}
   \begin{align}
   A \text{~is nonexpansive}
    \siff~& \innp{(x-y)-(u-v),2(u-v)-(x-y)}
    \nonumber
    \\
    &\ge -\tfrac{1}{2}\normsq{2(u-v)-(x-y)}
    \\
    \siff~&-\normsq{x-y}-2\normsq{u-v}+3\innp{x-y,u-v}
    \nonumber\\ 
   &\ge
   -2\normsq{u-v}-\tfrac{1}{2}\normsq{x-y}+2\innp{x-y,u-v}
   \\
 \siff ~& \innp{x-y, u-v}\ge \tfrac{1}{2}\normsq{x-y}
 \\
 \siff ~& J_A \text{~is~} \tfrac{1}{2}\text{-strongly monotone}
 \\
  \siff ~& R_A \text{~is~} \text{ monotone},
   \end{align}
  \end{subequations}
where the last equivalence follows from
\cref{lem:JA:RA:corr}\ref{lem:JA:RA:corr:iii}.

\ref{p:corres:A:RA:i}:
Let
$\{(x,u),(y,v)\}\subseteq X \times X$.
Using \cref{lem:gr:RA} we have
$
\{(x,u),(y,v)\}\subseteq\gra R_A$
 $\siff \big\{\big(\tfrac{1}{2}(x+u), \tfrac{1}{2}(x-u)\big),
 \big(\tfrac{1}{2}(y+v), \tfrac{1}{2}(y-v)\big)\big\}\subseteq\gra A$.
%
%
 Let
$\{(x,u),(y,v)\}\subseteq\gra R_A$.
Applying
\cref{cor:non:av:ch}\ref{cor:non:av:ch:ii}
with $(T,x,y)$ replaced by $\big(A, \tfrac{1}{2}(x+u),\tfrac{1}{2}(y+v)\big)$
and \cref{rem:con:nonexp},
we learn that
\begin{subequations}
   \begin{align}
   A \text{~is~} \alpha\text{-averaged~}
    \siff~&
    2\alpha\innp{\tfrac{1}{2}((x-y)-(u-v)),u-v}
    \ge (1-2\alpha)\normsq{u-v}
    \\
    \siff~&\alpha\innp{x-y,u-v}-\alpha\normsq{u-v}
    \ge (1-2\alpha)\normsq{u-v}
    \\
 \siff ~& \tfrac{\alpha}{1-\alpha}\innp{x-y, u-v}\ge \normsq{u-v},
   \end{align}
  \end{subequations}
 equivalently
  $R_A$ is $\tfrac{1-\alpha}{\alpha}$-cocoercive.
  \ref{p:corres:A:RA:i:d}:
Apply  \ref{p:corres:A:RA:i}
with $\alpha=\tfrac{1}{2}$.
\end{proof}

\begin{rem}
\cref{p:corres:A:RA}\ref{p:corres:A:RA:ii}
generalizes
the conclusion of \cite[Proposition~5.3]{Gisel17}.
Indeed, if $\beta>0$ we have
$A$ is $\beta$-cocoercive,
equivalently $R_A$
is $\tfrac{1}{\beta+1}$-averaged.
\end{rem}

\section{$\rho$-monotone and $\rho$-comonotone linear
operators}
\label{sec:linear}
Let
$A\in \RR^{n\times n}$
 and set
 $A_{s}=\tfrac{A+A\tran}{2}$.
 In the following we use
  $\lam_{\min}(A)$ and $\lam_{\max}(A)$ to
denote the smallest and largest eigenvalue of
$A$, respectively, provided all eigenvalues of $A$ are real.
\begin{prop}
\label{prop:linear:hypo}
Suppose that $A\in \RR^{n\times n}$.
Then
the following hold:
\begin{enumerate}
\item
\label{prop:linear:hypo:i}
$A$
is $\rho$-monotone $\siff$
$\lam_{\min}(A_{s})\ge \rho$.
\item
\label{prop:linear:hypo:ii}
$A$
is $\rho$-comonotone $\siff$
$\lam_{\min}(A_{s}-\rho A\tran A)\ge 0$.
\end{enumerate}
\end{prop}
\begin{proof}
Let $ x\in \RR^n$.
\cref{prop:linear:hypo:i}:
$A$ is $\rho$-monotone
$\siff$
$\innp{x,Ax}\ge \rho \normsq{x}$
$\siff$
$\innp{x,(A-\rho \Id)x}\ge 0$
$\siff$
$\innp{x,(A-\rho \Id)_sx}\ge 0$
$\siff$
$\innp{x,(A_s-\rho \Id)x}\ge 0$
$\siff$
$A_s-\rho \Id\succeq 0$
$\siff$
$A_s\succeq\rho \Id$
$\siff$
$\lam_{\min}(A_s)\ge \rho$.
\cref{prop:linear:hypo:ii}:
$A$ is $\rho$-comonotone
$\siff$
$\innp{x,Ax}\ge \rho \normsq{Ax}$
$\siff$
$\innp{x,(A_s-\rho A\tran A)x}\ge 0$
$\siff$
$A_s-\rho A\tran A\succeq 0$
$\siff$
$\lam_{\min}(A_s-\rho A\tran A)\geq 0$.
\end{proof}

\begin{example}
  \label{ex:linear:rmono:rcomono}
Suppose that $N\colon X\to X$
is continuous and linear 
such that $N^*=-N$
 and $N^2=-\Id$.
 Then $N$ is nonexpansive.
 Moreover,
let $\lambda\in \left[0,1\right[$,
set $T_\lambda=(1-\lambda)\Id+\lambda N$
 and set
 $A_\lambda=(T_\lambda)^{-1}-\Id$.
Then the following hold:
\begin{enumerate}
  \item
\label{ex:linear:rmono:rcomono:i}
  We have
\begin{equation}
A_\lambda=\tfrac{\lambda}{(1-\lambda)^2+\lambda^2}
\big((1-2\lambda)\Id-N\big).
\end{equation}
\item
\label{ex:linear:rmono:rcomono:ii}
$ A_\lambda$ is $\rho$-monotone
with optimal $\rho
=\tfrac{\lambda(1-2\lambda)}{\lam^2+(1-\lam)^2}$.
\item
\label{ex:linear:rmono:rcomono:iii}
$ A_\lambda$ is $\rho$-comonotone
with optimal $\rho=
\tfrac{1-2\lambda}{2\lambda}
$.
\end{enumerate}
\end{example}
\begin{proof}
Let $x\in X$.
Then $\norm{Nx}^2
=\innp{Nx,Nx}
=\innp{x,N^*Nx}
=\innp{x,-N^2x}
=\innp{x,x}=\normsq{x}
$.
 Hence $N$ is nonexpansive;
in fact, $N$ is an isometry.
Now set
\begin{equation}
B_\lambda = \tfrac{\lambda}{(1-\lambda)^2+\lambda^2}
\big((1-2\lambda)\Id-N\big).
\end{equation}
\cref{ex:linear:rmono:rcomono:i}:
We have
\begin{subequations}
\begin{align}
(\Id+B_\lambda)T_\lambda
&=\left(\Id+\tfrac{\lambda}{(1-\lambda)^2
+\lambda^2}\big((1-2\lambda)\Id-N\big)\right)
\big((1-\lambda)\Id+\lambda N\big)
\\
&=\tfrac{1}{(1-\lambda)^2+\lambda^2}
\big((1-\lambda)\Id-\lambda N\big)\big((1-\lambda)\Id+\lambda N\big)
\\
&=\tfrac{1}{(1-\lambda)^2+\lambda^2}
\big((1-\lambda)^2\Id-\lambda^2N^2\big)=\Id.
\end{align}
\end{subequations}
Similarly, one can show that
$T_\lambda(\Id+B_\lambda)=\Id$
and the conclusion follows.

\cref{ex:linear:rmono:rcomono:ii}:
Using
\cref{ex:linear:rmono:rcomono:i},
we have
\begin{subequations}
\begin{align}
\innp{x, A_\lambda x}
&= \frac{\lambda}{(1-\lambda)^2+\lambda^2}
\big((1-2\lambda)\norm{x}^2-\innp{Nx,x}\big)\\
&= \frac{\lambda(1-2\lambda)}{(1-\lambda)^2+\lambda^2}
\norm{x}^2.
  \label{sub:eq:lin}
\end{align}
\end{subequations}

\cref{ex:linear:rmono:rcomono:iii}:
Using
\cref{ex:linear:rmono:rcomono:i},
we have
\begin{subequations}
\begin{align}
\normsq{A_\lambda x}
&= \frac{\lambda^2}{((1-\lambda)^2+\lambda^2)^2}
\big((1-2\lam)^2\normsq{x}+\normsq{Nx}\big)
\\
&= \frac{\lam^2}{((1-\lam)^2+\lam^2)^2}
\big((1-2\lam)^2+1\big)\normsq{x}.
\end{align}
\end{subequations}
Therefore, combining with
\cref{sub:eq:lin}
we obtain
\begin{subequations}
\begin{align}
\innp{x, A_\lambda x}
&=
\frac{(1-2\lam)((1-\lam)^2+\lam^2)}{\lam ((1-2\lam)^2+1)}
\cdot
\frac{\lam^2((1-2\lam)^2+1)}{((1-\lam)^2+\lam^2)^2}
\normsq{x}
\\
&=\frac{(1-2\lam)((1-\lam)^2+\lam^2)}{\lam ((1-2\lam)^2+1)}
\normsq{A_\lambda x},\\
&=\frac{1-2\lambda}{2\lambda}
\normsq{A_\lambda x},
\end{align}
\end{subequations}
and the conclusion follows.
\end{proof}

\section{Hypoconvex functions}
\label{sec:hypocon}
In this section, we apply results in the previous sections
to characterize proximal mappings of
hypoconvex functions.
We shall assume that
$f\colon X\to \left]-\infty,+\infty\right]$ is a proper
lower semicontinuous function minorized
by a concave quadratic function:
$\exists \nu\in\RR, \beta\in  \RR, \alpha\geq 0$ such that
$$(\forall x\in X)\quad
f(x)\geq -\alpha\|x\|^2-\beta\|x\|+\nu.$$
For $\mu>0$, the Moreau envelope of $f$ is defined
by
$$e_{\mu}f(x)=\inf_{y\in X}\Big(f(y)+\frac{1}{2\mu}\|x-y\|^2\Big),$$
and the associated proximal mapping $\prox_{\mu f}$ by
\begin{equation}
\prox_{\mu f}(x)
=\underset{y\in X}{\argmin}\Big(f(y)+\frac{1}{2\mu}\|x-y\|^2\Big),
\end{equation}
where $x\in X$. We shall use $\partial f$ for 
the subdifferential mapping from convex analysis.
\begin{definition}\label{def:abst}
An abstract subdifferential $\pars$ associates a subset $\pars
f(x)$ of $X$ to
$f$ at $x\in X$, and it 
satisfies the following properties:
\begin{enumerate}
\item\label{i:p1} $\pars f=\partial f$ if $f$ is a proper lower semicontinuous convex function;
\item\label{i:p1.5} $\pars f=\nabla f$ if $f$ is continuously differentiable;
\item\label{i:p2} $0\in\pars f(x)$ if $f$ attains a local minimum at $x\in\dom f$;
\item\label{i:p3} for every $\beta\in\RR$, 
$$\pars\Big(f+\beta\frac{\|\cdot-x\|^2}{2}\Big)=\pars f +\beta (\Id-x).$$
\end{enumerate}
\end{definition}
The Clarke--Rockafellar subdifferential,
Mordukhovich subdifferential, and Fr\'echet subdifferential all
satisfy Definition~\ref{def:abst}\ref{i:p1}--\ref{i:p3},
see, e.g., \cite{Clarke90}, \cite{Mord06, Mord18},
so they are $\pars$. Related but different abstract subdifferentials have been
used in \cite{aussel95, ioffe84, thibault95}.

Recall that $f$
is $\tfrac{1}{\lambda}$-hypoconvex (see \cite{Rock98, Wang2010}) if
\begin{equation}
  f((1-\tau)x+\tau y)\le (1-\tau) f(x)
  +\tau f(y) +\frac{1}{2\lambda}\ \tau(1-\tau)
  \norm{x-y}^2,
\end{equation}
for all $(x,y)\in X\times X$ and $\tau\in \left]0,1\right[$.

\begin{proposition}\label{p:sub:for}
 If $f\colon X\to \left]-\infty,+\infty\right]$ is a proper
 lower semicontinuous $\frac{1}{\lambda}$-hypoconvex function, then
\begin{equation}\label{e:hypoc:sub}
\pars f=\partial\Big(f+\frac{1}{2\lambda}\|\cdot\|^2\Big)-\frac{1}{\lambda}\Id.
\end{equation}
Consequently, for a hypoconvex function the Clarke--Rockafellar,
Mordukhovich,
and Fr\'echet subdifferential operators all coincide.
\end{proposition}
\begin{proof}
 For the convex function $f+\frac{1}{2\lambda}\|\cdot\|^2$, apply Definition~\ref{def:abst}\ref{i:p1} and
 \ref{i:p3} to obtain
$$\partial\Big(f+\frac{1}{2\lambda}\|\cdot\|^2\Big)
=\pars\Big(f+\frac{1}{2\lambda}\|\cdot\|^2\Big)
=\pars f+\frac{1}{\lambda}\Id$$
from which \eqref{e:hypoc:sub} follows.
\end{proof}

Let $f^*$ denote the Fenchel conjugate of $f$.
The following result is well known in $\RR^n$, see, e.g.,
\cite[Exercise~12.61(b)(c), Example~11.26(d)~and~Proposition~12.19]{Rock98}, and \cite{Wang2010}.
In fact, it also holds in a Hilbert space.

\begin{prop}\label{prop:hypo:func}
The following are equivalent:
\begin{enumerate}
  \item
  \label{prop:hypo:func:i}
  $f$ is $\tfrac{1}{\lambda}$-hypoconvex.
  \item
    \label{prop:hypo:func:ii}
  $f+\tfrac{1}{2\lam}
  \normsq{\cdot}$ is convex.
  \item
    \label{prop:hypo:func:iii}
  $\Id+\lambda \pars f$ is maximally monotone.
  \item
    \label{prop:hypo:func:vi}
$(\forall \mu\in \left] 0,\lam\right[)$
  $\prox_{\mu f}$ is $\lambda/(\lambda-\mu)$-Lipschitz
continuous with
\begin{equation}
  \label{eq:prox:res}
  \prox_{\mu f}
=J_{\mu\pars f}
=(\Id+\mu
\pars f )^{-1}.
\end{equation}

 \item
    \label{prop:hypo:func:v}
$(\forall \mu\in \left] 0,\lam\right[)$
  $\prox_{\mu f}$ is
  single-valued and continuous.
\end{enumerate}
\end{prop}

 \begin{proof}
 ``\ref{prop:hypo:func:i}$\Leftrightarrow$\ref{prop:hypo:func:ii}":
 Simple algebraic manipulations.

 ``\ref{prop:hypo:func:ii}$\Rightarrow$\ref{prop:hypo:func:iii}": As
 $$\partial\Big(f+\frac{1}{2\mu}\|\cdot\|^2\Big)=
 \pars\Big(f+\frac{1}{2\mu}\|\cdot\|^2\Big)=\pars f +\frac{1}{\mu}\Id$$
 is maximally monotone, $\Id+\mu\pars f$ is maximally monotone.

 ``\ref{prop:hypo:func:iii}$\Rightarrow$\ref{prop:hypo:func:vi}": By Definition~\ref{def:abst}\ref{i:p2} and \ref{i:p3},
 $y\in \prox_{\mu f}(x)$ implies that
 $$0\in\pars\Big(f(y)+\frac{1}{2\mu}\|y-x\|^2\Big)=\pars f(y)+\frac{1}{\mu}(y-x).$$
 Thus, one has
 \begin{equation}\label{e:prox:abst}
 (\forall x\in X)\ \prox_{\mu f}(x)\subseteq (\Id+\mu \pars f)^{-1}(x).
 \end{equation}
  Using
 $$\Id+\mu\pars f=\frac{\lambda-\mu}{\lambda}\Big(\Id+\frac{\mu}{\lambda-\mu}(\Id+\lambda\pars f)\Big)$$
 yields
 $$(\Id+\mu\pars f)^{-1}=J_{A}\circ\Big(\frac{\lambda}{\lambda-\mu}\Id\Big),$$
 where $A=\frac{\mu}{\lambda-\mu}(\Id+\lambda\pars f)$ is maximally monotone by the assumption.
 Since
 $J_{A}$ is nonexpansive on $X$, $(\Id+\mu\pars f)^{-1}$ is
 $\lambda/(\lambda-\mu)$-Lipschitz.
 Together with \eqref{e:prox:abst}, we obtain $\prox_{\mu f}=(\Id+\mu\pars f)^{-1}.$

 ``\ref{prop:hypo:func:vi}$\Rightarrow$\ref{prop:hypo:func:v}": Clear.

``\ref{prop:hypo:func:v}$\Rightarrow$\ref{prop:hypo:func:ii}": Let $x\in X$ and let $\mu\in \left]0,
\lambda\right[$.
  We have
 \begin{equation}
 \label{e:conj}
 e_{\mu}f(x)=\frac{1}{2\mu}\|x\|^2-\Big(f+\frac{1}{2\mu}\|\cdot\|^2\Big)^*\Big(\frac{x}{\mu}\Big),
 \end{equation}
 and $e_{\mu}$ is locally Lipschitz, see, e.g., \cite[Proposition 3.3(b)]{Jourani14}.
By \cite[Proposition 5.1]{Bernard05}, \ref{prop:hypo:func:v}
implies that $e_{\mu}f$ is Fr\'echet differentiable
with $\nabla e_{\mu}f=\mu^{-1}(\Id-\prox_{\mu f})$.
Then $\big(f+\frac{1}{2\mu}\|\cdot\|^2\big)^*$ is
Fr\'echet differentiable by \eqref{e:conj}. It follows from
\cite[Theorem 1]{Stromberg11} that $ f+\frac{1}{2\mu}\|\cdot\|^2$ is convex.
Since this hold for every $\mu\in ]0,\lambda[$, \ref{prop:hypo:func:ii} follows.
 \end{proof}

We now provide a new refined characterization of
hypoconvex functions in terms of the cocoercivity of
their proximal operators; equivalently,
of the conical nonexpansiveness of the
displacement mapping of
their proximal operators.
\begin{theorem}
  \label{prop:hypocon:av}
Let $\mu\in \left] 0,\lambda\right[$.
Then the following are equivalent.
\begin{enumerate}
  \item
  \label{prop:hypocon:av:i}
  $f$ is $\tfrac{1}{\lambda}$-hypoconvex.
  \item
  \label{prop:hypocon:av:ii}
  $\Id-\prox_{\mu f}$ is
  $\tfrac{\lambda}{2(\lambda-\mu)}$-\conic.
\item
\label{prop:hypocon:av:iii}
$\prox_{\mu f}$ is $\tfrac{\lambda-\mu}{\lambda}$-cocoercive.
\end{enumerate}
\end{theorem}
\begin{proof}
``\ref{prop:hypocon:av:i}$\siff$\ref{prop:hypocon:av:ii}":
Using $0<\tfrac{\mu}{\lambda}<1$ we have
\begin{align*}
  &\qquad\text{$f$ is $\tfrac{1}{\lambda}$-hypoconvex}
  \\
&\siff
\text{$\Id+\lambda \pars f$ is maximally monotone}
\tag{\text{by \cref{prop:hypo:func}}}
\\
&\siff  \text{$\tfrac{\mu}{\lambda}\Id+ \mu\pars f$
is maximally monotone}
\\
&\siff  \text{$\mu\pars f$
is maximally $\bigl(-\tfrac{\mu}{\lambda}\bigr)$-monotone}
\\
&\siff  \text{$(\mu\pars f)^{-1}$
is maximally $\bigl(-\tfrac{\mu}{\lambda}\bigr)$-comonotone}
\tag{\text{by \cref{lem:bmax:inv}}}
\\
&\siff
\text{$J_{(\mu\pars f)^{-1}}$ }
\text{is $\tfrac{\lambda}{2(\lambda-\mu)}$-\conic}
\tag{\text{by \cref{cor:eq:nexp:-0.3}\ref{cor:eq:nexp:-0.3:0}}}
\\
&\siff
\text{$ \Id - J_{\mu\pars f}$ }
\text{is $\tfrac{\lambda}{2(\lambda-\mu)}$-\conic}
\tag{\text{by  \cref{prop:gen:min}\ref{prop:gen:min:i}
}}
\\
&\siff  \text{$\Id-\prox_{\mu f} $
is $\tfrac{\lambda}{2(\lambda-\mu)}$-\conic}
\tag{\text{by \cref{eq:prox:res}}}.
\end{align*}
``\ref{prop:hypocon:av:ii}$\siff$\ref{prop:hypocon:av:iii}":
Use  \cref{lem:coco:conc}.
\end{proof}

\begin{corollary}
  \label{cor:Lips:grad}
Suppose that
$f\colon X\to \RR$ is Fr\'echet differentiable
such that
$\grad f$ is Lipschitz
with a constant $1/\lambda$.
Then the following hold:
\begin{enumerate}
  \item
  \label{cor:Lips:grad:i}
    $\Id +\lambda \grad f$ is maximally monotone.
  \item
  \label{cor:Lips:grad:ii}
  $f$ is $\tfrac{1}{\lambda}$-hypoconvex.
  \item
  \label{cor:Lips:grad:iii}
  $f+\tfrac{1}{2\lambda}
  \normsq{\cdot}$ is convex.
  \item
  \label{cor:Lips:grad:iv}
$(\forall \mu\in \left] 0,\lambda \right[)$
$\prox_{\mu f}$ is
single-valued.
\item
\label{cor:Lips:grad:v}
$(\forall \mu\in \left] 0,\lambda \right[)$
$\prox_{\mu f}$
is $\tfrac{\lambda-\mu}{\lambda}$-cocoercive.
\item
\label{cor:Lips:grad:vi}
$(\forall \mu\in \left] 0,\lambda \right[)$
$
\prox_{\mu f}
=J_{\mu\pars f}
=(\Id+\mu
\grad f )^{-1}.
$
\item
\label{cor:Lips:grad:vii}
$(\forall \mu\in \left] 0,\lambda  \right[)$
$\Id-\prox_{\mu f}$ is
$\tfrac{\lambda}{2(\lambda-\mu)}$-\conic.
\end{enumerate}
\end{corollary}
\begin{proof}
Definition~\ref{def:abst}\ref{i:p1.5}
  implies that $(\forall x\in X)$
  $\pars f(x)=\{\grad f(x)\}$.
\ref{cor:Lips:grad:i}:
Indeed,
$\lambda \grad f$ is nonexpansive.
Now the conclusion follows from
\cite[Example~20.29]{BC2017}.
\ref{cor:Lips:grad:ii}--\ref{cor:Lips:grad:vii}:
Combine \ref{cor:Lips:grad:i} with
\cref{prop:hypo:func}
and \cref{prop:hypocon:av}.
\end{proof}

Finally, we give two examples to illustrate our results.

\begin{example}
Suppose that $X=\RR$.
Let $\lambda>0$ and set, for every $\lambda $,
$f_\lambda
\colon x\mapsto
\exp(x)-\tfrac{1}{2\lambda}x^2 $.
Then
$f$ is $\tfrac{1}{\lambda}$-hypoconvex
by \cref{prop:hypo:func},
$f'_\lambda\colon x\mapsto
\exp(x)
-\tfrac{x}{\lambda}$,
and
we have $\Id+\lambda f'_\lambda=\lambda \exp $ is
maximally monotone.
Moreover, for every $\mu\in \left]0,\lambda\right ]$
we have
\begin{subequations}
\begin{align}
\prox_{\mu f_\lambda}(x)
&=\bigl( \Id+\mu f'_\lambda\bigr)^{-1}(x)
=\bigl((1-\tfrac{\mu}{\lambda})\Id+\mu\exp\bigr)^{-1}(x)
\label{se:1}
\\
&=\begin{cases}
\ln\bigl(\tfrac{x}{\mu}\bigr) , &\text{if~~}\mu=\lambda;\\
\tfrac{\lambda x}{\lambda-\mu}
-\operatorname{Lambert}  W
 \bigl(\tfrac{\lambda\mu\exp(\lambda x/(\lambda-\mu ))}{\lambda-\mu }\bigr),
&\text{if $\mu\in \left]0,\lambda\right[$},
\end{cases}
\end{align}
\end{subequations}
where the first identity in \cref{se:1}
follows from \cref{cor:Lips:grad}\ref{cor:Lips:grad:vi}.
\end{example}

\begin{example}
  Let $D$ be a nonempty closed convex subset of $X$,
  let $\lambda>0$ and set, for every $\lambda $,
  $f_\lambda
  =
  \iota_D -\tfrac{1}{2\lambda}\norm{\cdot}^2 $.
  Then
  $f$ is $\tfrac{1}{\lambda}$-hypoconvex
  by \cref{prop:hypo:func},
   and
  $\pars  f_\lambda
  =N_D-\tfrac{1}{\lambda}\Id$
  by
  Proposition~\ref{p:sub:for}.
Moreover,  for every $\lambda>0$,
  we have $\Id+\lambda \pars  f_\lambda=N_D$ is
  maximally monotone.
Finally, using \cref{eq:prox:res}
  and \cite[Example~23.4]{BC2017}
  we have for every $\mu\in \left]0,\lambda\right [$
  \begin{subequations}
  \begin{align}
  \prox_{\mu f_\lambda}
  &=\bigl( \Id+\mu \pars  f_\lambda\bigr)^{-1}
  =\bigl((1-\tfrac{\mu}{\lambda})\Id+\mu N_D\bigr)^{-1}\\
  &=\bigl((1-\tfrac{\mu}{\lambda})(\Id+ N_D)\bigr)^{-1}
  =P_D\circ \bigl(\tfrac{\lambda}{\lambda-\mu}\Id\bigr).
  \end{align}
  \end{subequations}
  In particular, if $D$ is a closed convex cone
  we learn that $  \prox_{\mu f_\lambda}=\tfrac{\lambda}{\lambda-\mu}P_D$.
\end{example}

\section*{Acknowledgment}
HHB, WMM, and XW were partially
supported by the Natural Sciences and Engineering Council
of Canada.


\small

\begin{thebibliography}{999}
\seppfive

 \bibitem{aussel95} D.\ Aussel, J.-N.\ Corvellec, and M.\ Lassonde,
 Mean value property and subdifferential criteria for lower semicontinuous functions,
  \emph{Transactions of the American Mathematical Society}~347 (1995), 4147--4161.

\bibitem{BC2017}
H.H.\ Bauschke and P.L.\ Combettes,
\emph{Convex Analysis and Monotone
Operator Theory in Hilbert Spaces},
Second Edition,
Springer, 2017.
 \bibitem{BMW2}
 H.H.\ Bauschke, S.M.\ Moffat, and X.\ Wang,
 Firmly nonexpansive mappings and maximally monotone operators:
 correspondence and duality,
 \emph{Set-Valued and Variational Analysis}~20 (2012), 131--153.

\bibitem{Bernard05}
F.\ Bernard and L.\ Thibault, Prox-regular functions in Hilbert spaces,
\emph{Journal of Mathematical Analysis and Applications} 303 (2005), 1--14.


  \bibitem{Borwein50}
 J.M.\ Borwein,
 Fifty years of maximal monotonicity,
 \emph{Optimization Letters}~4 (2010), 473--490.
 \bibitem{Brezis}
 H.\ Brezis,
 \emph{Operateurs Maximaux Monotones et
 Semi-Groupes de Contractions dans les Espaces de Hilbert},
 North-Holland/Elsevier, 1973. 


 \bibitem{BurIus}
 R.S.\ Burachik and A.N.\ Iusem,
 \emph{Set-Valued Mappings and Enlargements
 of Monotone Operators},
 Springer-Verlag, 2008.

\bibitem{Clarke90}
F.H.\ Clarke, \emph{Optimization and Nonsmooth Analysis},
Second Edition, Classics in Applied Mathematics,
SIAM,
Philadelphia, PA, 1990.


\bibitem{Clarke75}
F.H.\ Clarke,
Generalized gradients and applications,
\emph{Transactions of the American Mathematical Society}~205
(1975), 247--262.



 \bibitem{Comb96}
 P.L.\ Combettes,
 The convex feasibility problem in image recovery,
 \emph{Advances in Imaging and Electron Physics}~25 (1995), 155--270.
 \bibitem{Comb04}
 P.L.\ Combettes,
 Solving monotone inclusions via compositions of nonexpansive averaged
 operators,
 \emph{Optimization}~53 (2004), 475--504.
\bibitem{CombPenn04}
P.L.\ Combettes and T.\ Pennanen,
Proximal methods for cohypomonotone operators,
\emph{SIAM Journal on Control and Optimization}~43 (2004),
731--742.
\bibitem{EckBer}
J.\ Eckstein and D.P.\ Bertsekas,
On the Douglas--Rachford splitting method
and the proximal point algorithm for maximal monotone
operators,
\emph{Mathematical Programming}~55 (1992), 293--318.
\bibitem{Gisel17}
P.\ Giselsson,
Tight global linear convergence rate bounds
for Douglas--Rachford splitting,
\emph{Journal of Fixed Point Theory
and Applications},
DOI 10.1007/s11784-017-0417-1

\bibitem{ioffe84} A.D.\ Ioffe, 
Approximate subdifferentials and applications I: The finite-dimensional theory,
\emph{Transactions of the American Mathematical Society}~281 (1984), 389--416.

\bibitem{Jourani14}
A.\ Jourani, L.\ Thibault, and D.\ Zagrodny, Differential properties of the Moreau envelope,
\emph{Journal of Functional Analysis} 266 (2014), 1185--1237.

\bibitem{krans}
M.A.\ Krasnosel'ski\u{\i},
Two remarks on the method of successive
approximations,
\emph{Uspekhi Matematicheskikh Nauk}~10 (1955),
 123--127.
\bibitem{Mann}
 W.R.\ Mann,
 Mean value methods in iteration,
 \emph{Proceedings of the
 American  Mathematical Society}~4 (1953),
 506--510.
\bibitem{Minty}
G.J.\ Minty,
Monotone (nonlinear) operators in Hilbert spaces,
\emph{Duke Mathematical Journal} 29 (1962), 341--346.

\bibitem{Mord18}
B.S.\ Mordukhovich,
\emph{Variational Analysis and Applications},
Springer Monographs in Mathematics,
Springer, 2018.

\bibitem{Mord06}
B.S.\ Mordukhovich,
\emph{Variational Analysis and Generalized Differentiation I},
Basic Theory,
Springer, 2006.

\bibitem{MV18}
W.M.\ Moursi and L.\ Vandenberghe,
Douglas--Rachford splitting for a Lipschitz
continuous
and a strongly monotone operator.
\url{https://arxiv.org/pdf/1805.09396.pdf}.

\bibitem{PhanDao18}
H.M.\ Phan and M.N.\ Dao,
Adaptive Douglas--Rachford splitting
algorithm for the sum of two operators,
\url{https://arxiv.org/pdf/1809.00761.pdf}.

\bibitem{Rock76}
R.T.\ Rockafellar,
Monotone operators and the proximal point algorithm,
\emph{SIAM Journal on Control and Optimization},
14 (1976), 877--898.



\bibitem{Rock98}
R.T.\ Rockafellar and R.J-B.\ Wets,
\emph{Variational Analysis},
Springer-Verlag, 
corrected 3rd printing, 2009.


%
 \bibitem{Simons1}
 S.\ Simons,
 \emph{Minimax and Monotonicity},
 Springer-Verlag,
 1998.

 \bibitem{Simons2}
 S.\ Simons,
 \emph{From Hahn-Banach to Monotonicity},
 Springer-Verlag,
 2008.

\bibitem{Stromberg11}
T.\ Str\"omberg, Duality between Fr\'echet differentiability and strong convexity,
\emph{Positivity} 15 (2011), 527--536.

\bibitem{thibault95} L.\ Thibault and D.\ Zagrodny, Integration of subdifferentials of lower semicontinuous functions on Banach spaces, \emph{Journal of Mathematical Analysis and Applications}~189 (1995), 33--58.

\bibitem{Wang2010}
X.\ Wang,
On Chebyshev functions and Klee functions,
\emph{Journal of Mathematical Analysis and
Application} 368(2010), 293--310.

 \bibitem{Zeidler2a}
 E.\ Zeidler,
 \emph{Nonlinear Functional Analysis and Its Applications II/A:
 Linear Monotone Operators},
 Springer-Verlag, 1990.

 \bibitem{Zeidler2b}
 E.\ Zeidler,
 \emph{Nonlinear Functional Analysis and Its Applications II/B:
 Nonlinear Monotone Operators},
 Springer-Verlag, 1990.

\end{thebibliography}
\end{document}